\documentclass[twoside,11pt,leqno]{article}
\newcommand{\oR}{{\mathbb R}}

\newcommand{\oN}{{\mathbb N}}
\newcommand{\xx}{\mathbf{x}}
\newcommand{\yy}{\mathbf{y}}

\newcommand{\cfa}{\mathbf{a}}
\newcommand{\EE}{{\mathbb E}}

\newcommand{\cX}{{\cal{X}}}
\newcommand{\1}{{\mathbf{1}}}

\newtheorem{res}{Lemma}
\newtheorem{cor}{Corollary}
\newtheorem{thm}{Theorem}

\newenvironment{proof}{\noindent{\bf Proof:}}{\hfill $\square$ \\}
\usepackage{graphicx,amssymb,amsmath}

\begin{document}
\bibliographystyle{plain}
\thispagestyle{empty}
\begin{center}
{\Large\bf Optimal decision rules for marked point process models}\\[.4in]

\noindent
{\large M.N.M. van Lieshout}\\[.1in]
\noindent
{\em CWI, P.O.~Box 94079, NL-1090 GB  Amsterdam, The Netherlands \\[.1in]
Department of Applied Mathematics,
University of Twente, P.O.~Box 217 \\ NL-7500 AE Enschede,
The Netherlands}\\[.1in]
%\today
\end{center}

\begin{abstract} \noindent
We study a Markov decision problem in which the state space is the set of 
finite marked point configurations in the plane, the actions represent 
thinnings, the reward is proportional to the mark sum which is discounted 
over time, and the transitions are governed  by a birth-death-growth process. 
We show that thinning points with large marks is optimal when births follow 
a Poisson process and marks grow logistically. Explicit values for the thinning
threshold and the discounted total expected reward over finite and infinite 
horizons are also provided. When the points are required to respect a hard 
core distance, upper and lower bounds on the discounted total expected reward 
are derived.  \\

\end{abstract}

\section{Introduction}
\label{S:intro}

The classic Markov decision process \cite{Bert95,FeinSchw02,Pute94} on a 
finite state space ${\cal{X}}$ and action set $A$ is defined as follows. 
Write $A(x)$ for the subset of $A$ which contains all actions that may be 
taken in state $x\in{\cal{X}}$. Then, a policy $\phi$ is a procedure for the 
selection of an action at each decision epoch $i\in \oN_0$. Such a policy 
could be random or deterministic, and in principle take into account the 
entire history of the process. Often though, one may restrict attention to 
the class of deterministic Markov policies. Such a policy  $\phi = 
(\phi_i)_{i=0}^\infty$ is a sequence of mappings $\phi_i: {\cal{X}} \to A$
that, at time $i$, assign an action $a = \phi_i(x) \in A(x)$ to the current 
state $x$. In doing so, a direct reward $r(x,a)$ is earned and a probability 
mass function $p(\cdot|x,a)$ on ${\cal{X}}$ governs the next state of the 
process. Being Markovian, only the current state and action are important; 
the past history is irrelevant. A policy $(\phi_i)_{i}$ is said to be 
stationary if its members $\phi_i$ do not depend on the time $i$. 

Let $(X_i, Y_i)$ denote the stochastic process of states $X_i$ and 
actions $Y_i$. Write $\EE_x^\phi$ for its expectation when the
initial state $X_0=x$ and the transitions are driven by policy $\phi$.
Then an optimal policy maximises the discounted total expected reward
\begin{equation}
v_\alpha^\phi(x) =
\EE_x^\phi\left[ \sum_{i=0}^{\infty} \alpha^i r(X_i, Y_i) \right],
\label{e:discount}
\end{equation}
$0 \leq \alpha < 1$.
The reward function is usually assumed to be bounded, in which case 
(\ref{e:discount}) is well-defined. When the state and action spaces 
are both finite, it is well known \cite[Theorem~5.5.3b]{Pute94} that it 
suffices to consider only Markov policies and, by \cite[Theorem~6.2.10]
{Pute94}, one may restrict oneself even further to the class of Markov 
policies that are deterministic and stationary. The maximal discounted
total expected reward can be found by policy iteration \cite[Theorem~6.4.2]
{Pute94} or value iteration, also known as successive approximation or 
dynamic programming.

When the cardinality of the state or action space is infinite, policy
iteration is not guaranteed to converge in a finite number of steps.
The dynamic programming approach on the other hand is amenable to
generalisation to more general state and action spaces. Results in
this direction include \cite{BertShrev78,FeinLewi07,HernLass96,Scha93}. 
The tutorial by Feinberg \cite{Fein16} provides an exhaustive overview with 
particular emphasis on inventory control problems.

In this paper, we concentrate on the case where the state space ${\cal{X}}$
consists of finite simple marked point patterns in two-dimensional
Euclidean space. Markov decision theory using spatial point process models 
has found many applications in mobile network optimisation. However, the 
role of the point process is auxiliary in that it is used to model the 
spatial distribution of users, base stations and so on, from which coverage
probabilities and other performance characteristics of the network can be 
calculated \cite{BaccBart09,Leeetal20,Luetal21,Khloetal15}. 
Spatial point process models are also convenient in multi-target tracking 
\cite{Lies08} and their void probabilities or divergence measures can form 
the basis for observer trajectory optimisation \cite{Bearetal17}. 

Our focus of interest here is to assume that the actions operate directly
on the point process. More precisely, we assume that, at decision epoch $i$,
an action $\phi_i(\xx)$ maps $\xx$ into a subset of $\xx$. In other words, the
action set $A(\xx)$ is the finite power set of $\xx$. When the decision to 
retain a point or not is based on the mark or the inter-point distances, 
it can be interpreted as a (mark-)dependent thinning \cite{Mate86,Myll09}.
The set-up described above is appropriate for harvesting problems in forestry 
\cite{Pret09}. Here, the classical strategy is to use discretised stand based 
growth tables and dynamic programming \cite{Ronn03}. Point pattern based 
policies have been rarer due to `a lack of models and to difficulties in 
selecting trees to be removed' \cite{PukkMiin98} and tend to be simulation 
based \cite{Franetal20,Pukketal15,RensSark01,Rensetal09}. One example is
German thinning,  which enhances natural selection by picking trees whose
diameter at breast height is at most $d$ and fells a fraction of them. 
More formally, if each $X_i$ consists of tree locations marked by diameter 
at breast height, German thinning fells a fraction of the set
\(
\{ (x,m) \in X_i: m \leq d \}
\)
and the Markov transition kernel governs the growth of the remaining trees 
(for example using the logistic growth curve or extensions such as the Richards 
curve \cite{Rich59}) as well as natural births and deaths (e.g.\ a hardcore model 
\cite{KellRipl76}, the asymmetric soft core models of \cite{Lies09} or the 
dynamic models of \cite{RensSark01}).  French thinning is similar, except that 
a fraction of trees with large rather than small sizes is removed to 
stimulate forest rejuvenation. In either case, picking a policy amounts to 
choosing the level $d$. Simulations suggest that French thinning might be
the better strategy \cite{Franetal20}.

The paper's plan is as follows. In Section~\ref{S:Verhulst}, we study a  
decision process in which the actions consist of deleting a subset of the current 
points and the reward is proportional to the marks. The stochastic process that 
governs the dynamics is a birth-and-death process with independent deaths and  
a Poisson process of births; the marks grow logistically. We calculate the 
discounted total expected reward function over finite and infinite horizons 
and show that French thinning is an optimal policy. An explicit expression for 
the mark threshold $d$ is derived too. In Section~\ref{S:hardcore}, 
we move on to allow interaction between the points and replace the Poisson 
birth process by one in which no point is allowed to come too close to
another point. In this setting, we provide upper and lower bounds on the 
discounted total expected reward function over finite and infinite horizons. The 
tightness of the bounds is investigated by means of some simulated examples.
We conclude by mentioning some topics for further research.

\section{Marked Poisson process model with logistic growth}
\label{S:Verhulst}

\subsection{Definition of the model}
\label{S:Defs}

To define a Markov decision process \cite[Section~2.3.2]{Pute94},
let the state space $\cX$ consists of finite simple marked point patterns
on a compact set $W \subset \oR^2$ with marks in $L =[0,K]$ 
for some $K > 0$. When $\cX$ is equipped with the Borel $\sigma$-algebra
of the weak topology, by the discussion below \cite[Prop~9.1.IV]{DaleVere88}, 
$\cX$ is Polish. When at time $i \in \oN_0$ the process is in state $\xx$, 
a thinning action is carried out, resulting in a new state $\cfa$ that 
consists of all retained points $\cfa \subset \xx$. Thus, the action space 
$A(\xx)$ is finite and contains all subsets of $\xx$. 
Define a stationary reward function $r_i(\xx, \cfa) = r(\xx, \cfa)$ by
\begin{equation}
 r(\xx, \cfa ) = R \sum_{ (x,m) \in \xx \setminus \cfa } m, \quad 
 \xx \in \cX, \cfa \subset \xx.
\label{e:reward}
\end{equation}
Thus, the reward is proportional to the sum of the marks of all 
removed points. When $R > 0$, $r(\cdot, \cdot) \geq 0$. Since the mark 
content in an $\oR^+$-marked point process is a random variable 
\cite[Proposition~6.4.V]{DaleVere88}, $r$ is well-defined.

Upon taking action $\cfa$ in state $\xx$, the dynamics that lead to
the next state are modelled as a birth-death-growth process. Specifically,
the marks of the retained points $(x,m) \in \cfa$ grow according to the 
well-known logistic model that was proposed around 1840 by Verhulst and 
Quetelet \cite{Rich59}. In this model, when the mark at time $0$
is $m_0 > 0$, the mark at time $n \in \oN \cup \{ 0 \}$ is
\begin{equation}
\label{e:gn}
g^{(n)}(m_0) =
 \frac{K}{ 1 + \left( \frac{K}{m_0} - 1 \right) e^{-\lambda n} }.
\end{equation}
By convention, $g^{(n)}(0) = 0$. The parameter $\lambda > 0$ governs 
the rate of growth and $K \geq m_0 \geq 0$ is an upper bound on the size. 
In combination with independent births and deaths, the next state is 
defined by the following dynamics:
\begin{itemize}
\item delete $\xx \setminus \cfa$;
\item independently of other points, let each $(x_i, m_i) \in \cfa$ die with 
      probability $p_d \in (0,1)$ (natural deaths) and otherwise grow to
      $(x_i, g^{(1)}(m_i))$ as in (\ref{e:gn});
\item add a Poisson process on $W$ with intensity $\beta > 0$ and mark its
      points independently according to a probability measure $\nu$ on $[0,K]$.
\end{itemize}

Write $(X_i, Y_i)_{i=0}^\infty$ for the sequence of successive states $X_i$ and 
actions $Y_i$. A randomised policy $\phi = (\phi_i)_{i=0}^\infty$ is a sequence 
of conditional probability kernels 
$\phi_i(\cdot | X_0$, $Y_0, \dots$, $ X_{i-1}$, $Y_{i-1}$, $X_i)$
on $A$ to generate $Y_i$ based on the history of the process such that 
$\phi_i( A(\xx_i) | \xx_0, \cfa_0, \dots, \xx_i) = 1$. If the policy is 
Markov and deterministic, $Y_i$ is simply a function of $X_i$, and one may 
write $Y_i= \phi_i(X_i)$. Then, for $0 \leq \alpha < 1$, the infinite horizon 
$\alpha$-discounted total expected reward function (\ref{e:discount}) under 
policy $\phi = (\phi_i)_{i=0}^\infty$ with initial state $X_0 = \xx$ is 
\begin{equation}
  v_\alpha^\phi(\xx) = \EE^\phi \left[
    \sum_{i=0}^\infty \alpha^i  \left(
       R \sum_{(x,m)\in X_i \setminus Y_i} m  \right)
     \mid X_0 = \xx \right].
\label{e:value}
\end{equation}
%\cite[p.~217]{BertShrev78}. 

The following Lemma shows that the model is well-defined for the 
birth-death-grow dynamics defined above.

\begin{res} \label{L:reward-finite}
The infinite horizon $\alpha$-discounted total expected reward function 
$v_\alpha^\phi(\xx)$, $\xx \in \cX$, defined in (\ref{e:value})
is finite for all $0 \leq \alpha < 1$, all $R>0$ and all policies $\phi$.
\end{res}

\begin{proof}
Pick $\xx \in \cX$ and write $n(\xx) < \infty$ for its cardinality. 
Since the growth function (\ref{e:gn}) is bounded by $K$, 
\[
 \EE\left[ \sum_{(x,m)\in X_0 \setminus Y_0} m \mid X_0 = \xx \right] 
 \leq K n(\xx).
\]
For $i > 0$,  $X_i$ is the union of survivors from $\xx$,
from subsequent generations starting with $X_1\setminus X_0$ up to
$X_{i-1} \setminus X_{i-2}$ and points born in the last decision epoch.
Therefore, recalling the birth and death dynamics,
\[
  \EE \left[ \sum_{(x,m)\in X_i \setminus Y_i} m | X_0 = \xx \right] 
 \leq 
   K n(\xx) (1-p_d)^i \\
   + 
   K  \beta |W| \sum_{k=0}^{i-1} (1-p_d)^k
\]
where $|W|$ denotes the area of $W$. Hence
\[
v_\alpha^\phi(\xx) \leq R K n(\xx) \sum_{i=0}^\infty \alpha^i (1-p_d)^i 
+ R K \beta |W| \sum_{i=1}^\infty \alpha^i \sum_{k=0}^{i-1} (1-p_d)^k.
\]
For all $p_d \in (0,1)$, the first series in the right hand side converges
to $ 1 / ( 1 - \alpha (1-p_d))$. Since
\[
  \sum_{i=1}^\infty \alpha^i \sum_{k=0}^{i-1} (1-p_d)^k 
  = \sum_{i=1}^\infty \alpha^i \frac{1-(1-p_d)^i}{p_d} \leq
  \frac{1}{p_d} \sum_{i=1}^\infty \alpha^i < \infty
\]
for all $p_d \in (0,1)$, $v_\alpha^\phi(\xx)$ is finite.
\end{proof}

The reward function $r$ itself is not bounded, so the (N) regime of 
\cite[Chapter~9]{BertShrev78} applies.
%since the number of points may not be bounded. 

\subsection{Optimal policy and reward}
\label{S:French}

The optimal $\alpha$-discounted total expected reward $v^*_\alpha(\xx)$ is 
defined  as the supremum of the $v_\alpha^\phi(\xx)$ over all policies, 
including randomised ones. In this section, we will show that French thinning 
is optimal and give an explicit expression for the corresponding reward. 

By \cite[Proposition~9.1]{BertShrev78}, the supremum in the definition of
$v^*_\alpha(\xx)$ may be taken over the class of Markov policies, and,
by \cite[Proposition~9.8]{BertShrev78}, satisfies the equation
\begin{equation} \label{e:Bellman}
v^*_\alpha(\xx) = \max_{\cfa \subset \xx} \left\{
  R \sum_{(x,m) \in \xx \setminus \cfa} + \alpha \EE \left[ v^*_\alpha(X) \mid \xx, \cfa \right]
\right\}
\end{equation}
where $X$ is distributed according to the one step birth-death-growth
dynamics from state $\xx$ under action $\cfa$. Observe that the optimality 
equations (\ref{e:Bellman}), $\xx \in \cX$, are not sufficient conditions 
for $v^*_\alpha$. Nevertheless, $v^*_\alpha(\xx)$ can be calculated as the 
limit of an iterative procedure \cite[Proposition~9.14]{BertShrev78} known 
as the dynamic programming algorithm. Set $v_0(\xx) = 0$ for all $\xx \in \cX$ 
and set $n=1$.  Define, for every $\xx \in \cX$, 
\[
v_n(\xx) = \max_{\cfa \subset \xx} \left\{
   R \sum_{(x,m)\in \xx \setminus \cfa } m 
+ \alpha \EE \left[ v_{n-1}(X)  \mid \xx, \cfa \right] \right\}.
\]
Then set $n = n + 1$ and repeat. This algorithm converges to $v^*_\alpha(\xx)$ 
as $n\to \infty$ by \cite[Proposition~9.14]{BertShrev78} but -- in general -- 
is of little help in constructing an optimal policy, let alone a stationary one. 
Given a stationary policy $\phi$, a necessary and sufficient condition for it 
to be optimal is \cite[Prop.~9.13]{BertShrev78}
\begin{equation}
v_{\alpha}^{\phi}(\xx) = \max_{\cfa \subset \xx} \left\{
R \sum_{(x,m) \in \xx \setminus \cfa } +
\alpha \EE\left[
v_{\alpha}^{\phi}(X) | \xx, \cfa \right]
\right\}.
\label{e:optimalDiscount}
\end{equation}

For our model, the dynamic programming algorithm does suggest 
an optimal deterministic and stationary Markov policy. 

\begin{thm}
\label{t:simple}
Consider the Markov decision process with state space $\cX$, action
spaces $A(\xx) = \{ \yy \in \cX: \yy \subset \xx\}$, $\xx \in \cX$, 
reward function (\ref{e:reward}) with $R > 0$, and birth-death-growth 
dynamics based on independent deaths with probability $p_d \in (0,1)$, 
a Poisson birth process with intensity $\beta > 0$ marked independently 
according to probability measure $\nu$ on $[0,K]$ for $K > 0$ and logistic 
growth function (\ref{e:gn}).
Then, for $0\leq \alpha < 1$,
\[
  v^*_\alpha(\xx) =  R \beta | W| \sum_{k=1}^\infty \alpha^k \int_0^K s(m) \, d\nu(m)
  + R \sum_{(x,m)\in \xx} s(m),
\]
where $|W|$ is the area of $W$ and
\[
s(m) =  \sup_{n\in \oN_0}  \left\{ \frac{ K \alpha^n ( 1 - p_d)^n }{
1 + \left( \frac{K}{m} - 1 \right) e^{-\lambda n} } \right\}, 
\quad m \in [0, K].
\]
Furthermore, the optimal $\alpha$-discounted total expected reward  
corresponds to a French thinning that removes all points with a mark 
that is at least
\[
   d^*_\alpha =  \sup_{n\in \oN_0} \left\{ \frac{K}{1 - e^{-n\lambda}} \left(
    \alpha^n ( 1 - p_d)^n - e^{-n \lambda} \right) \right\}.
\]
\end{thm}

For $\alpha = 1$, the total expected reward $v^*_1(\xx)$ is infinite.

\medskip

%\begin{ex}
%For the special case $\alpha = 0$, it is best to fell all trees,
%that is, $a^*_0 = 0$.  Furthermore, $s(m) = g^{(0)}(m) = m$, and
%\[
%v_0^*(\xx) = R \sum_{(x,m)\in\xx} m.
%\]
%The same value function is obtained if all trees die in one time unit, 
%$p_d = 1$ (and also $a^*_0 = 0$, $s(m) = m$).
%\end{ex}

\begin{proof}
After initialising $v_0(\xx) = 0$ for all $\xx \in \cX$, clearly the 
optimal expected reward at time $0$ is
\(
v_1(\xx) =  R \sum_{(x,m)\in \xx} m,
\)
which is attained for action $\cfa = \emptyset$, or, in other words,
by removing all points with mark greater than or equal to 
$d_1 = 0$. The proof proceeds by induction. Set, for $n\in\oN$,
\begin{equation}
d_n = \max\left\{ 0,  K \frac{ \alpha ( 1- p_d) - e^{-\lambda}}
  {1 - e^{- \lambda}}, \,  \dots, \,
   K \frac{ \alpha^{n-1} ( 1- p_d)^{n-1} - e^{-(n-1)\lambda}}
      {1 - e^{- (n-1)\lambda}} 
\right\}
\label{e:an}
\end{equation}
and suppose that the optimal $\alpha$-discounted expected reward 
over $n$ steps is attained by French thinning at level $d_n$ and given by
\begin{equation}
  v_n(\xx)  = 
  R \beta |W| 
  \sum_{k=1}^{n-1} \alpha^k \int_0^K s_{n-k}(m) \, d\nu(m) 
   + R \sum_{(x,m)\in \xx} s_n(m)
\label{e:vn}
\end{equation} 
where,  for $1\leq k\leq n$,
\[
s_k(m) =  \max\left\{ 
    m, \alpha \left(1-p_d \right) g^{(1)}(m), \, \dots, \, 
    \alpha^{k-1} \left(1-p_d \right)^{k-1} g^{(k-1)}(m) 
\right\}.
\]
Now, for $n+1$, the optimal finite horizon $\alpha$-discounted 
expected reward is 
\[
  v_{n+1}(\xx) = \max_{ \cfa \subset \xx} \left\{
     R \sum_{(x,m)\in \xx\setminus \cfa} m 
   +  \alpha \EE \left[ v_n(X) \mid \xx, \cfa \right] 
  \right\}.
\]
By the induction assumption, the discounted expectation 
$\alpha \EE \left[ v_n(X) \mid \xx, \cfa \right]$ is the sum of
\[
  \alpha R \beta |W| \sum_{k=1}^{n-1} \alpha^k \int_0^K s_{n-k}(m) \, d\nu(m)
  =  R \beta |W| \sum_{k=2}^{n} \alpha^k \int_0^K s_{n+1-k}(m) \, d\nu(m)
\]
and contributions from the points in $\cfa$ that survive a decision epoch
as well as from points born in the interval between time $n$ and $n+1$. These contributions are, respectively,
\[
  \alpha R \sum_{(x,m) \in \cfa} \left(1 - p_d \right) s_n( g(m) )
\]
and, using the Campbell--Mecke formula \cite[Section~6.1]{DaleVere88},
\[
 \alpha R \beta |W| \int_0^K s_n(m) \, d\nu(m).
\]
The optimal policy assigns a point $(x,m) \in \xx$ to $\xx\setminus \cfa$
if and only if 
\(
 m \geq \alpha(1-p_d) s_n ( g^{(1)}(m) ).
\)
By the induction assumption and (\ref{e:gn}), this is the case if and only if
\begin{equation}
m \geq \alpha ^k( 1-p_d)^k  g^{(k)}(m) 
\Leftrightarrow 
   m \geq K \frac{\alpha^k (1-p_d)^k - e^{-k\lambda}}{1 - e^{-k\lambda}}
   \label{e:threshold}
\end{equation}
for all integers $1 \leq k \leq n$. Consequently, 
$d_{n+1}$ has the required form. For this allocation rule, the reward is
\(
\max\left\{ m,  \alpha \left( 1 - p_d \right) s_n( g^{(1)}(m) ) \right \} = s_{n+1}(m)
\)
and the induction step is complete.

Next, let $n$ go to infinity and fix $m \in [0, K]$. Then $s(m)$
is finite for all $p_d \in (0,1)$ and $0\leq \alpha < 1$. Additionally,
$\lim_{n\to\infty} s_n(m) = s(m)$. Thus, for 
any $\xx \in \cX$,  
\[
  R \sum_{(x,m)\in\xx} s_n(m) \to  R \sum_{(x,m)\in\xx} s(m)
\]
as $n\to \infty$. Furthermore,
\[
  \sum_{k=1}^{n-1} \alpha^k \int_0^K s_{n-k}(m) \, d\nu(m)
\to
\sum_{k=1}^\infty \alpha^k \int_0^K s(m) \, d\nu(m), \quad n\to \infty,
\]
because of dominated convergence applied to the doubly indexed
sequence $a_{k,n}$ defined by
\(
\1\left\{ k \leq n-1 \right\} \alpha^k \int s_{n-k} \, d\nu.
\)
In conclusion, for each $\xx \in \cX$, $\lim_{n\to\infty} v_n(\xx) =
v^*_\alpha(\xx)$, the optimal $\alpha$-discounted total expected reward
\cite[Proposition~9.14]{BertShrev78}, and $v^*_\alpha(\xx)$ has the
claimed form.

To complete the proof, we need to show that $v^*(\xx)$ is attained by
the stationary deterministic policy that retains all points with mark smaller
than $d^*_\alpha$. Denote its infinite horizon $\alpha$-discounted total 
expected reward by 
\[
 v_\alpha^{d^*}(\xx) = \EE \left[
 R \sum_{i=0}^\infty \alpha^i 
 \sum_{(x,m)\in X_i} m \, \1\{ m \geq d^*_\alpha \} \mid X_0 = \xx \right]
\]
and focus on the contributions of each generation of points. 
A point $(x,m) \in \xx$, the initial generation, yields a reward
$R \, \alpha^n (1-p_d)^n g^{(n)}(m)$ precisely when $g^{(n-1)}(m)$ is 
less than $d^*_\alpha$ but $g^{(n)}(m) \geq d^*_\alpha$. 
Since, as in (\ref{e:threshold}),
\(
g^{(n)}(m) \geq d^*_\alpha
\)
if and only if 
\[
g^{(n)}(m) \geq \alpha^k ( 1 - p_d  )^k g^{(n+k)}(m)
\]
for all $k \in \oN_0$, 
we conclude that every point of $\xx$ contributes $R \, s(m)$.
The points that are born in the first decision epoch (generation $1$) 
yield the same total expected reward, but this is discounted by $\alpha$
due to the later birth date. Similarly, the total expected reward of points 
belonging to the second generation is discounted by $\alpha^2$, and so on. 
Tallying up, the  $\alpha$-discounted total expected reward of generations
$k = 1, 2, \dots$ is
\[
 R \beta |W| \sum_{k=1}^\infty \alpha^k \int_0^K s(m) \, d\nu(m)
\]
on application of the Campbell--Mecke formula. Finally add the contribution
from the initial generation to conclude that the threshold $d^*_\alpha$ 
defines an optimal policy. Condition (\ref{e:optimalDiscount})
is readily verified.
\end{proof}

As a by-product, the proof of Theorem~\ref{t:simple} derives the optimal
$\alpha$-discounted total expected reward (\ref{e:vn}) for finite time 
horizons too, and French thinning with threshold (\ref{e:an}) is an 
optimal policy. The suprema in $s(m)$ and $d^*_\alpha$ are attained,
which can be seen by considering the limit for $n\to\infty$. 

\section{Hard core models with logistic growth}
\label{S:hardcore}

\subsection{Bounds for the optimal discounted total expected reward}
\label{S:bounds}

In this section, we refine the Poisson model of the previous section to
the case where births are governed by a hard core process. Thus, the 
state space $\cX_K$ consists of all finite simple marked point patterns
on a compact set $W$ in the plane that contain no pair $\{ x_1, x_2 \}$ 
such that $|| x_1 - x_2 || \leq K$ with marks in $L = [0, K]$. For the 
motivating example from forestry in which the marks correspond to the 
diameter at breast height, the condition ensures that all trees can grow 
to their maximal size. 

As in Section~\ref{S:Defs}, when at time $i \in \oN_0$ the process is in 
state $\xx$, a thinning action is carried out, resulting in a new state 
$\cfa$ that consists of all retained points. The reward is defined
in (\ref{e:reward}).

The dynamics are modified in such a way that the hard core is respected. 
Specifically, suppose that action $\cfa$ is taken in state $\xx \in \cX_K$. 
The next state is then governed by the following birth-death-growth process:
\begin{itemize}
\item delete $\xx \setminus \cfa$;
\item independently of other points, let each $(x_i, m_i) \in \cfa$ die with
      probability $p_d \in (0,1)$ and otherwise grow to
      \(
      (x_i, g(m_i) )
      \)
      for some bounded, continuous function $g: [0,K] \to [0,K]$
      satisfying $m \leq g(m)$ for $m \in [0,K]$;
%\[
%      \left(x_i,
%          \frac{K}{
%                    1 + e^{-\lambda} \left( \frac{K}{m_i} - 1 \right) }
%      \right), \quad \lambda > 0;
%\]
\item add a hard core process on $W$ with hard core distance $K$ and 
      intensity $\beta > 0$; mark its points independently according to 
      a probability measure $\nu$ on $[0,K]$ and remove all points that
      fall within distance $K$ to a point in $\cfa$.
\end{itemize}

In this framework, the reward function is bounded since the hard core 
condition implies an upper bound on the number of points that can be 
alive at any time. We are therefore in the (D) regime of 
\cite[Chapter~9]{BertShrev78}. 

For $\xx\in \cX_K$, define $v^*_\alpha(\xx)$ as the supremum of 
(\ref{e:value}) over all policies $\phi$. By \cite[Proposition~9.1]
{BertShrev78} it suffices to consider Markov policies only, and 
$v^*_\alpha(\xx)$ is the limit of the dynamic programming algorithm
\cite[Proposition~9.14]{BertShrev78}. The optimality condition 
(\ref{e:optimalDiscount}) applies. Moreover, since the action 
sets are finite, Corollary 9.17.1 in \cite{BertShrev78} guarantees
the existence of an optimal deterministic stationary policy.
%and the maximum in 
%[
% \max_{\cfa \subset \xx} \left\{
%     R \sum_{(x,m) \in \xx \setminus \cfa } +
%      \alpha \EE\left[
%     v^*_{\alpha}(X) 
%      | \xx, \cfa \right]
% \right\}
%\]
%is achieved \cite[Corollary~9.12.1]{BertShrev78}. 
An explicit expression seems hard to obtain. However, the following 
bounds are available.

\begin{thm} 
Consider the Markov decision process with state space $\cX_K$, action
spaces $A(\xx) = \{ \yy \in \cX_K: \yy \subset \xx \}$, $\xx \in \cX_K$, 
reward function (\ref{e:reward}) with $R > 0$, and birth-death-growth 
dynamics based on independent deaths with probability $p_d \in (0,1)$, 
a hard core birth process with intensity $\beta > 0$ marked independently 
according to probability measure $\nu$ on $[0,K]$ for $K > 0$ and growth 
function $g$. Write $g^{(n)}(m)$ for the $n$-fold composition of $g$.

For $\alpha \in [0,1)$, initialise $v_0(\xx) = 0$ for all $\xx\in \cX_K$.  
Define, for $n\in\oN$ and $\xx\in\cX_K$,
\[
v_n(\xx) = \max_{\cfa \subset \xx} \left\{ 
R \sum_{(x,m)\in\xx\setminus\cfa} m + \alpha \EE\left[
    v_{n-1}(X) \mid \xx, \cfa \right] \right\}.
\]
Then $\tilde v_n(\xx) \leq v_n(\xx) \leq \hat v_n(\xx)$ where
\begin{eqnarray*}
\tilde v_n(\xx) & = & R \sum_{(x,m)\in \xx} \tilde s_n(x,m) 
+ R \beta \sum_{k=1}^{n-1} \alpha^k \int_{W} \int_0^K 
   \tilde s_{n-k}(w,l) \, d\nu(l) dw \\
\hat v_n(\xx) & = & R \sum_{(x,m)\in \xx} \hat s_n(m) 
+ R \beta |W| \sum_{k=1}^{n-1} \alpha^k 
\int_0^K \hat s_{n-k}(l) \, d\nu(l)  \\
\end{eqnarray*}
with $\tilde s_0 = \hat s_0 = 0$ and, for $n\in\oN$,
\[
\hat s_n(m) = \max \left\{ 
 m, \alpha (1-p_d) g^{(1)}(m), \,\dots, \,
    \alpha^{n-1} (1-p_d)^{n-1} g^{(n-1)}(m)
\right\}
\]
and, writing $b(x,K)$ for the closed ball centred at $x$ with radius $K$,
\[
\tilde s_n(x,m) = \max\{ 
 m, \alpha (1-p_d) g^{(1)}(m) - \alpha K \beta | b(x,K) \cap W |, \, \dots,
\]
\[ 
 \alpha^{n-1} (1-p_d)^{n-1} g^{(n-1)}(m) 
 - \alpha K \beta | b(x,K) \cap W | \sum_{i=0}^{n-2} \alpha^i (1-p_d)^i
\}.
\]
\label{t:hardcore}
\end{thm}

When the growth function is logistic, 
\begin{eqnarray*}
\tilde s_n(x,m) & = & \max_{i=0, \dots, n-1} \left\{ 
\frac{K \alpha^i ( 1 - p_d )^i}{1 + \left(\frac{K}{m} - 1 \right) e^{-\lambda i}}
-\alpha K \beta | W \cap b(x,K) | 
\frac{ 1 - \alpha^{i}(1-p_d)^i }{ 1 - \alpha ( 1 - p_d) }
\right\};
\\
\hat s_n(m) & = & \max_{ i = 0, \dots, n-1} \left\{ 
\frac{ K \alpha^i ( 1 - p_d )^i }
{1 +  \left(\frac{K}{m} - 1 \right) e^{-\lambda i}}
\right\}.
\end{eqnarray*}

\mbox{}

\begin{proof}
The proof proceeds by induction. For $n=0$, evidently
$v_0 \leq \tilde v_0$. Assume that 
$\tilde v_k(\xx) \leq v_k(\xx) \leq \hat v_k(\xx)$
for all $k \leq n$ and all $\xx\in\cX_K$ and that $\tilde v_k$,
$\hat v_k$ have the required form. Since
\begin{equation}
v_{n+1}(\xx) = \max_{ \cfa \subset \xx } \left\{
   R \sum_{(x,m) \in \xx \setminus \cfa } m +
  \alpha \EE \left[ v_n(X) \mid \xx, \cfa \right]
\right\}
\label{e:vnext}
\end{equation}
and $v_n(X) \geq \tilde v_n(X)$, let us consider the expectation of 
$\tilde v_n(X)$ under the hard core birth-death-growth dynamics 
when action $\cfa$ is taken in state $\xx$. By the definition 
of $\tilde v_n$ and distinguishing between surviving and new-born points,
\begin{eqnarray*}
\EE \left[ \tilde v_n(X) \mid \xx, \cfa \right] & = &
 R \, \EE \left[ \sum_{(x,m) \in X} \tilde s_n(x,m) \mid 
    \xx, \cfa \right]
+ R \beta \sum_{k=1}^{n-1} \alpha^k \int_{W} \int_0^K
\tilde s_{n-k}(w,l) \, d\nu(l) dw  \\
& = &
R  \sum_{(x,m) \in \cfa} \left( 1-p_d \right) \tilde s_n(x,g^{(1)}(m) )
+ R \beta \sum_{k=1}^{n-1} \alpha^k \int_{W} \int_0^K
\tilde s_{n-k}(w,l) \, d\nu(l) dw  \\
& &
+ R \beta \int_{W} \int_0^K \tilde s_{n}(w,l) \,
 \1\{ w \not \in U_K(\cfa) \} \, d\nu(l) dw  
\end{eqnarray*}
where the symbol $U_K(\cfa)$ signifies the union of closed balls with 
radius $K$ around the points in $\cfa$. The calculation of the last term 
above relies on the Campbell--Mecke formula \cite[Section~6.1]{DaleVere88}. 
Now, the integral in the last line above can be written as
\[
R \beta \int_{W} \int_0^K \tilde s_n(w,l) \, d\nu(l) dw
- R \beta \int_{W} \int_0^K \tilde s_n (w, l)  \,
   \1\{ w \in U_K(\cfa \} \, d\nu(l) dw
\]
and is bounded from below by 
\begin{equation}
R \beta \int_{W} \int_0^K \tilde s_n(w,l) \, d\nu(l) dw
- R K \beta \sum_{(x,m)\in \cfa} \int_{W} \int_0^K
 \1\{ w \in b(x,K) \} \, d\nu(l)
\label{e:lowerbound}
\end{equation}
where the induction assumption is invoked for the inequality
$\tilde s_n \leq K$.
Next, return to (\ref{e:vnext}). The bound on
$\EE \left[ \tilde v_n(X) \mid \xx, \cfa \right]$ implies 
\[
v_{n+1}(\xx) \geq
 \max_{\cfa \subset \xx} \left\{ 
R \sum_{(x,m)\in\xx\setminus\cfa} m + \alpha \EE\left[
    \tilde v_{n}(X) \mid \xx, \cfa \right] \right\} \\
 \geq  \max_{\cfa \subseteq \xx}
\{
R \sum_{(x,m) \in \xx\setminus \cfa} m
+ 
\]
\[
\alpha R  \sum_{(x,m) \in \cfa} \left[ ( 1-p_d ) 
  \tilde s_n(x,g^{(1)}(m) )
- K \beta | b(x,K) \cap W | 
\right] 
 + R \beta \sum_{k=1}^{n} \alpha^k \int_{W} \int_0^K
  \tilde s_{n+1-k}(w,l) d\nu(l) dw .
\]
The policy that assigns $(x,m)$ to $\xx\setminus \cfa$ if and only if
\[
m \geq \alpha\left[ \left( 1-p_d \right) \tilde s_n(x, g^{(1)}(m)) 
  - K \beta | b(x,K) \cap W|  \right] 
\]
optimises the right hand side and, with 
\[
\tilde s_{n+1}(x,m) = \max\left\{
   m, \alpha \left( 1-p_d \right) \tilde s_{n}(x, g^{(1)}(m)) 
   - \alpha K \beta | b(x,K) \cap W | \right\},
\]
one sees that 
\[
v_{n+1}(\xx) \geq \tilde v_{n+1}(\xx) =
R \sum_{(x,m)\in \xx} \tilde s_{n+1}(x,m) + 
R \beta \sum_{k=1}^{n} \alpha^k \int_{W} \int_0^K
  \tilde s_{n+1-k}(w,l) \, d\nu(l) dw,
\]
an observation that completes the induction argument and 
therefore the proof of the lower bound.

For the upper bound $v_n \leq \hat v_n$, as in the
proof of Theorem~\ref{t:simple}, an induction 
proof applies 
based on $\hat s_n$ but with (\ref{e:lowerbound}) replaced by
the upper bound
\[
R \beta \int_{W} \int_0^K \hat s_n(w,l) \, d\nu(l) dw.
\]
\end{proof}

Over an infinite time horizon, the optimal $\alpha$-discounted
total expected reward is bounded by the same functional forms, 
which coincide if $\alpha = 0$.

\begin{cor}
The functions $\hat s_n$ and  $\tilde s_n$ defined in
Theorem~\ref{t:hardcore} take values in $[0,K]$ and 
increase monotonically to 
\[
\hat s(m) =
\sup_{n\in \oN_0} \left\{  \alpha^n (1-p_d)^n g^{(n)}(m)  \right\}, 
\quad m \in [0,K],
\]
and, for $x\in W$ and $m\in [0,K]$,
\[
\tilde s(x,m) = \sup_{n\in \oN_0} \left\{ \alpha^n (1-p_d)^n g^{(n)}(m) 
  - \alpha K | b(x,K) \cap W | \sum_{i=0}^{n-1}
  \alpha^i (1-p_d)^i \right\}.
\]
\end{cor}

\subsection{Simulation study}
\label{S:simu}

To assess the tightness of the bounds in Theorem~\ref{t:hardcore}, 
we calculated $\hat v_n(\xx)$ and $\tilde v_n(\xx)$ in two regimes, 
a dense one and a sparse one. For the inital pattern $\xx$, a sample
from a Strauss process \cite{KellRipl76} on $W = [0,5]^2$ with 
interaction parameter set to zero was chosen. The activity parameter
was set to give the required intensity: $\beta = 1.0$ in the sparse
regime and $\beta = 4.3$ in the dense regime. For the mark dynamics, we 
used a logistic growth function with $\lambda = 2$ and maximal size 
$K=0.1$; the initial marks were sampled from a Beta distribution with 
shape parameters $\lambda_1 = 2$ and $\lambda_2 = 20$.
The death rate was set to $p_d = 0.05$. Finally, we used 
discount factor $\alpha = 0.9$ and reward parameter $R=1$. 

%hardcore($\beta = 1$ respectively $\beta = 5$)
%which results in intensity

The results are plotted in Figure~\ref{F:simu}. The left
panels show the pattern $\xx$. In the right panels, the
solid lines are the graphs of $\hat v_n(\xx)$ as a function
of $n$, the dotted lines show $\tilde v_n(\xx)$ plotted 
against $n$. Integrals were estimated by the Monte Carlo 
method with $1,000$ samples. In the sparse regime, the 
approximation is quite good, for the denser regime, the 
gap between the two graphs is quite wide except for very 
small $n$. In both cases, the dynamic programming algorithm 
converges rapidly.

\begin{figure}
\includegraphics[width=5in]{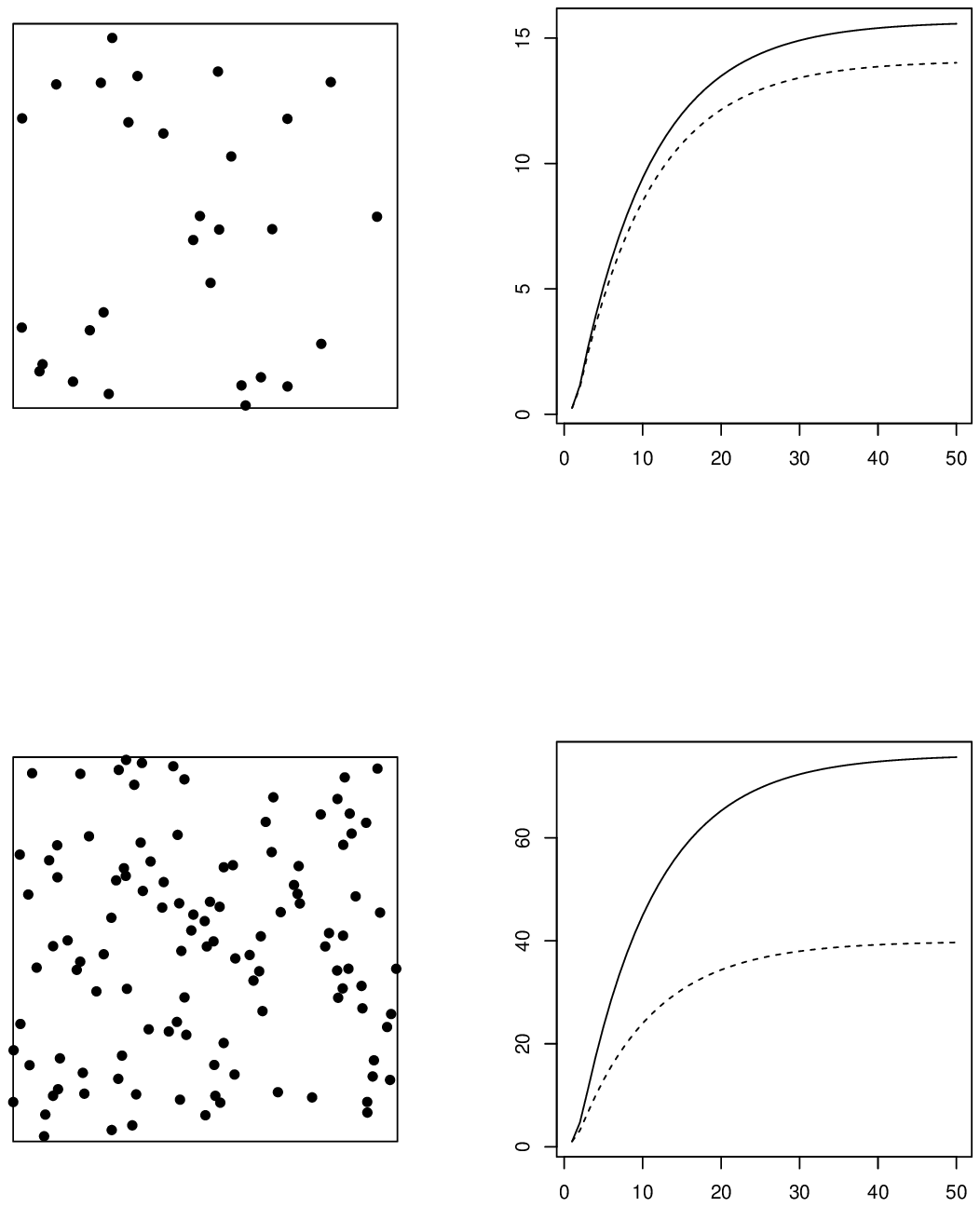}
\caption{Left panels: samples $\xx$ from a Strauss hard core 
process with intensity $\beta = 1.0$ (top) and $\beta = 4.3$ 
(bottom) on $[0,5]^2$. 
Right panels: graphs of $\hat v_n(\xx)$ (solid lines) and 
$\tilde v_n(\xx)$ (dotted lines) against $n$ for the  
birth-death-growth dynamics of Section~\ref{S:simu}.}
\label{F:simu}
\end{figure}

\section{Conclusion}

In this paper we considered optimal policies for Markov decision 
problems inspired by forest harvesting. We proved that French
thinning is optimal when births follow a Poisson process and 
marks grow logistically. When the points are required to respect
a hard core distance, we derived upper and lower bounds on the 
discounted total expected reward for general birth-death-growth 
dynamics. Although we focused on a homogeneous birth process, the 
results carry over to the case where the birth process is governed 
by some spatially varying intensity function. 

In future it would be of interest to study configuration dependent
asymmetric birth and growth models \cite{Lies08,Lies09,Rensetal09}.
Indeed, in a forestry setting, the growth of well-established, large 
trees may hardly be hampered by the emergence of saplings close by, 
while it would be harder for young and small trees to flourish near 
large ones. Moreover, the natural environment, such as the availability 
of nutrients, might play a role. Finally, refinements of the action 
space that allow for different thresholds in different mark strata 
could be investigated.

\end{document}